\documentclass[12pt]{amsart}

\usepackage{amsmath}
\usepackage{amsfonts}
\usepackage{amsthm}
\usepackage{amssymb}
\usepackage{enumerate}
\usepackage[english]{babel}
\usepackage[ansinew]{inputenc}
\usepackage[all]{xy}
\usepackage{color}

\theoremstyle{definition}

\newtheorem{defi}{Definition}[section]
\newtheorem{remark}[defi]{Remark}

\newtheorem{example}[defi]{Example}
\newtheorem{definition}[defi]{Definition}

\theoremstyle{plain}

\newtheorem{theorem}[defi]{Theorem}
\newtheorem{corollary}[defi]{Corollary}
\newtheorem{lemma}[defi]{Lemma}
\newtheorem{proposition}[defi]{Proposition}

\input xy
\xyoption{all}

\begin{document}

\title[Bumpy fundamental group scheme]{On the
bumpy fundamental group scheme}

\author[M. Antei]{Marco Antei}\thanks{The author thanks the project TOFIGROU (ANR-13-PDOC-0015-01)}

\address{Laboratoire J.A.Dieudonn\'e, UMR CNRS-UNS No 7351
Universit\'e de Nice Sophia-Antipolis, Parc Valrose,
06108 NICE Cedex 2, France}

\email{Marco.ANTEI@unice.fr}


\begin{abstract}
In this short paper we first recall the definition and the construction of the fundamental group scheme of a scheme $X$ in the known cases: when it is defined over a field and when it is defined over a Dedekind scheme. It classifies all the finite (or quasi-finite) fpqc torsors over $X$. When $X$ is defined over  a noetherian regular scheme $S$ of any dimension we do not know if such an object can be constructed. This is why we introduce a new category, containing the fpqc torsors, whose objects are torsors for a new topology. We prove that this new category is cofiltered thus generating a fundamental group scheme over $S$, said \textit{bumpy} as it may not be flat in general. We prove that it is flat when $S$ is a Dedekind scheme, thus coinciding with the \textit{classical} one.   \end{abstract}

\maketitle
\tableofcontents
\section{Introduction}

In his famous \cite{SGA1} Alexander Grothendieck constructs the étale fundamental group $\pi^{\text{\'et}}(X,\overline{x})$ of a scheme $X$, endowed with a geometric point $\overline{x}:Spec(\Omega)\to X$. It is a pro-finite group classifying the étale covers over $X$. Grothendieck elaborates for  $\pi^{\text{\'et}}(X,\overline{x})$ a specialization theory (\cite{SGA1}, Chapitre X) but it seems he is not entirely satisfied of it. Indeed at the end of the aforementioned tenth chapter he claims (here we freely translate from French\footnote{Une théorie satisfaisante de la
spécialisation du groupe fondamental doit tenir compte de la ``composante continue'' du ``vrai'' groupe fondamental, correspondant à la classification des revêtements principaux de groupe structural des groupes infinitésimaux ; moyennant quoi on serait en droit à
s'attendre que les ``vrais'' groupes fondamentaux des fibres géométriques d'un morphisme
lisse et propre $f:X\to Y$   forment un joli système local sur $X$, limite projective de
schémas en groupes finis et plats sur $X$. }) that \\

\emph{a satisfactory specialization theory
 of the fundamental group should consider the ``continuous component'' of the ``true'' fundamental group corresponding to the classification of principal homogeneous spaces with infinitesimal structure group; whereby one would have to expect that the ``true'' fundamental groups of the geometric fibres of a smooth and proper morphism
 $ f: X \to Y $ form a nice local system over $ X $, projective limit of 
finite and flat group schemes over  $X$.} \\

Though in a footnote in the same page he recalls us that\footnote{Cette conjecture extrêmement séduisante est malheureusement mise en défaut par un exemple inédit de M. Artin, déjà lorsque les fibres de $f$ sont des courbes algébriques de genre donné $g\geq 2$.} \\

\emph{this extremely seducing  conjecture is unfortunately contradicted by an unpublished example of M. Artin, even when the fibres of $f$ are algebraic curves of given genus $g\geq 2$} \\

the question of constructing a ``true'' fundamental group was neverthless very interesting. It is only at the end of the seventies that Madhav Nori, in his PhD thesis (see \cite{N1} and \cite{N2}), constructs it calling it ``the fundamental group scheme''.  It will be called again ``the true fundamental group'' in \cite{DelMil} but, at our knowledge, it will be the last time.  In \S \ref{sez:campo} we will recall his description, while in \S \ref{sez:Dede} we will give some details about the construction of the fundamental group scheme of a given scheme $X$ defined over a Dedekind scheme. In \S \ref{sez:ogni} we will finally suggest a new approach for higher dimensional base schemes: as we do not know if a fundamental object classifying all finite (or quasi-finite) fpqc torsors can be constructed, we introduce a new larger category whose objects are torsors for a new topology that we will describe. We prove that this new category is cofiltered thus generating a $S$-fundamental group scheme, said \textit{bumpy}, classifying all those new torsors. Though in general it may not be flat (which explains its name) we prove in \S \ref{sez:revisited} that it is flat at least when $S$ is a Dedekind scheme, thus coinciding with the \textit{classical} one. 

\section*{Acknowledgments}
We would like to thank A. Aryasomayajula, I. Biswas, A. Morye, A. Parameswaran and T. Sengupta for the wonderful meeting that they organised in both Tata Institute of Fundamental Research (TIFR, Mumbai) and University of Hyderabad (UoH) and for inviting me to actively participate to this event.

\section{Over a field}
\label{sez:campo}
The main references for the material contained in this sections are certainly  \cite{N1} and \cite{N2}, but also \cite{Sza}, Chapter 6. Other references will be provided when necessary for more detailed results.

\subsection{Tannakian description}
\label{sez:SGF11} 
Let  $k$ be a perfect\footnote{In these notes we are using Nori's setting, though it has been observed several times that his tannakian construction would hold for \emph{any} field $k$ provided $X$ satisfies the condition  $H^0(X,\mathcal{O}_X)=k$, which is automatically satisfied when $k$ is perfect.} field, $X$ a reduced and connected scheme, proper over $Spec(k)$, endowed with a $k$-rational point $x$. A vector bundle $V$ on $X$ is said to be \emph{finite} if there exist two polynomials  $f$ and $g$ with $f\neq g$ with non-negative integral coefficients such that $f(V)\simeq g(V)$, where, when we evaluate the polynomials, we have replaced the sum by the direct sum and the product by the tensor product. Now, let $C$ be a smooth and proper curve over $k$ and $W$ a vector bundle on $C$. Then $W$ is said to be \emph{semi-stable of degree $0$} if $deg(W)=0$ and for any sub-vector bundle $U\subset W$ we have $deg(U)\leq 0$.
\begin{definition}\label{defSS}
A vector bundle $V$ on $X$ is said to be Nori semi-stable\footnote{Of course this name has been given later; in Nori's works those vector bundles did not have a specific name.} if for any curve $C$ smooth and proper over $k$ and every morphism $j:C\to X$ the vector bundle $j^*(V)$ is semi-stable of degree $0$.
\end{definition}
When the characteristic of the field $k$ is $0$ then the category of finite vector bundles is tannakian  if we endow it with the trivial object  $\mathcal{O}_X$, the tensor product $\otimes_{\mathcal{O}_X}$ (of locally free sheaves) and the fibre functor $x^*$ (here $x:Spec(k)\to X$ is the given rational point). When the characteristic of the field $k$ is positive the category of finite vector bundle is in general not even abelian, this is why we need a bigger category containing the finite vector bundles. Nori proves that every finite vector bundle is indeed Nori semi-stable (cf. Definition \ref{defSS}): Nori's intuition has been to understand that the the abelian hull of the category of finite vector bundles inside the (abelian) category of Nori semi-stable vector bundles would be a good candidate to obtain an interesting tannakian category. He calls the objects of this category essentially finite vector bundles and the category itself is denoted by $EF(X)$, full subcategory of $\mathcal{C}oh(X)$, the category of coherent sheaves over $X$: it is thus the abelian category generated by finite vector bundles, their duals, their tensor products and their sub-quotients, full sub-category of the category of Nori semi-stable vector bundles. The following theorem summarize what we just said:

\begin{theorem}\label{teoNori1}
The category $EF(X)$ of essentially finite vector bundles endowed with the trivial object $\mathcal{O}_X$, the tensor product $\otimes_{\mathcal{O}_X}$ and the fibre functor $x^*$ is a tannakian category over $k$.
\end{theorem}

The duality theorem for tannakian categories (that we are tacitly assuming to be neutral) says that to $(EF(X),\mathcal{O}_X,\otimes_{\mathcal{O}_X},x^*)$ we can naturally associate an affine $k$-group scheme that we denote by $\pi(X,x)$. This is what we call the \emph{fundamental group scheme of $X$ in $x$.} Though in general a group scheme over a field is pro-algebraic (\cite{Wa}, \S 3.3) Nori proves that $\pi(X,x)$ is pro-finite.  Furthermore (cf. for instance \cite{DelMil}, Theorem 3.2) this theory provides a $\pi(X,x)$-torseur $\widehat{X}$, pointed in $\widehat{x}\in \widehat{X}_x(k)$ which is universal in the following sense: let $G$ be any finite $k$-group scheme and $Y\to X$ any $G$-torsor, pointed in $y\in Y_x(k)$, then there exists a unique $X$-morphism $\widehat{X}\to Y$ of torsors (i.e. commuting with the actions of $\pi(X,x)$ and $G$) sending  $\widehat{x}\mapsto y$. So we see that \emph{all} the finite torsors over $X$ are \emph{classified} by the fundamental group scheme, which is thus the object \emph{dreamt} by Grothendieck. When $k$ is algebraically closed then the group of rational points $\pi(X,x)(k)$ of $\pi(X,x)$   coincide with the étale fundamental group $\pi^{\text{\'et}}(X,x)$. We conclude this section with a recent result, obtained in \cite{BV} by Niels Borne and Angelo Vistoli who provide a very nice, useful and surprisingly simple  characterization  for essentially finite vector bundles:

\begin{proposition}Let $k$, $X$ and $x$ like before then a vector bundle $V$ on $X$ is essentially finite if and only if it is the kernel of a morphism between finite vector bundles.
\end{proposition}

%  Au foncteur  $$Rep_k(\pi(X,x))\simeq EF(X)\hookrightarrow Qcoh(X)$$ on associe de façon naturelle un   $\pi(X,x)$-torseur au dessus de $X$; pour voir ça on peut par exemple procéder comme suit:  à une représentation $V$ de  $\pi(X,x)$ on associe à travers l'équivalence de catégories tannakiennes $Rep_k(\pi(X,x))\simeq EF(X)$ 
%
%
%Ce torseur est appelé le $\pi(X,x)$-torseur universel.

\subsection{Description as pro-finite limit}
\label{sez:SGF12}       
There is a second construction for the fundamental group scheme of a scheme $X$ defined over a field $k$, again due to Nori (cf.\cite{N2}) and certainly inspired by the universal property satisfied by the universal torsor over $X$, as recalled in \ref{sez:SGF12}. This new construction is easier than the previous one, we do not use vector bundles or tannakian theory so it is in general more understandable by a wider audience. The assumptions on $X$ can be weakened a lot and at some point some people started to believe that the field $k$ could be replaced by a discrete valuation ring or, more generally, by a Dedekind scheme, but this will be the object of \S \ref{sez:Dede}. In this section $X$ will be a reduced and conected scheme defined over any field $k$ and $x\in X(k)$ will be a $k$-rational point. We consider the category $\mathcal{F}(X)$ where objects are finite torsors (i.e. their structural group is a $k$-finite group scheme) over $X$, pointed over $x$ and morphisms are morphisms commuting with the actions of their structural groups and sending a marked point to a marked point. We will often write $(Y,G,y)$ for a $G$-torsor $Y\to X$ pointed at $y$. Before stating the main result for this section let us briefly recall the notion of cofiltered category: a category $\mathcal{C}$ is cofiltered if it is non empty and if the following axioms are satisfied:
\begin{enumerate}
\item for any two objects  $A,B$ of $\mathcal{C}$ there exists an object $C$ of $\mathcal{C}$ and two morphisms $c_A:C\to A$ and $c_B:C\to B$;
\item for any two objects  $A,B$ for any two morphisms  $c_1:A\to B$ and $c_2:A\to B$ there exists an object $U$ of $\mathcal{C}$ and a morphism $u:U\to A$ such that $c_1\circ u=c_2\circ u$.
\end{enumerate}

We also observe that if the category  $\mathcal{C}$ has a final object and if for any triple of objects $A,B, C$ and for any pair of morphisms $f:A\to C$, $g:B\to C$ there exists, in $\mathcal{C}$, the fibre product $A\times_C B$ then the category $\mathcal{C}$ is cofiltered\footnote{Alternatively if $\mathcal{C}$ has a final object $U$, if there exist finite products (i.e. fibre products of objects over $U$) and if for any triple of objects $A,B, C$ and for any pair of morphisms $f:A\to C$, $g:B\to C$ there exists, in $\mathcal{C}$, a forth object  $D$ in $\mathcal{C}$ closing the diagram, then the category $\mathcal{C}$ is cofiltered}. This is a nice exercise left to the reader. We can now state the main result of this section, due to Nori:

\begin{theorem}\label{teoNori2}
Let $X$ be a reduced and connected scheme over any field $k$ and let $x\in X(k)$ be a $k$-rational point. Then the category $\mathcal{F}(X)$ is cofiltered.
\end{theorem}
\proof The strategy is the following: we take three objects of $\mathcal{F}(X)$  $(Y_i,G_i,y_i), i=0, ..., 2$ and two morphisms like in the following diagram
$$\xymatrix{(Y_1,G_1,y_1)\ar[rd] & & (Y_2,G_2,y_2)\ar[dl]\\ & (Y_0,G_0,y_0) & }$$
and we want to prove that this diagram has a fibre product in $\mathcal{F}(X)$. We chose as a candidate the triple given by $(Y_3,G_3,y_3):=(Y_1\times_{Y_0}Y_2,G_1\times_{G_0}G_2 ,y_1\times_{y_0}y_2 )$ and we need to prove that it is an object of $\mathcal{F}(X)$. It is not trivial to prove that $Y_1\times_{Y_0}Y_2$ is faithfully flat over  $X$ and it is at this point that we use our hypothesis on $X$, that we have assumed to be reduced and connected. Once we prove that $(Y_3,G_3,y_3)$ is a torsor over $X$ we are done.
\endproof

The importance of this result is that we can now wonder whether or not we can compute the projective limit of all the finite pointed torsors over $X$. But since $\mathcal{F}(X)$ is cofiltered and all the morphisms in $\mathcal{F}(X)$ are affine then this limit exists, what we obtain is again a universal triple  $(\tilde{X},\pi(X,x),\tilde{x})$; unsurprisingly we call again the pro-finite $k$-group scheme $\pi(X,x)$ the \emph{fundamental group scheme of $X$ in $x$} and  $\tilde{X}$ the universal $\pi(X,x)$-torsor (pointed at $\tilde{x}$). The fact that no confusion can arise is a consequence of the following

\begin{remark}Let $k$ be a perfect field, $X$ a reduced and connected scheme, proper over $Spec(k)$, endowed with a $k$-rational point such that $H^0(X,\mathcal{O}_X)=k$ then the tannakian construction introduced in \S \ref{sez:SGF11}  and the construction described by the Theorem  \ref{teoNori2} give rise to the same fundamental group scheme $\pi(X,x)$. 
\end{remark}

\section{Over a Dedekind scheme}
\label{sez:Dede}
As we pointed out in \S \ref{sez:SGF12} the pro-finite construction is more likely to be extended to a much more general situation, for instance when we replace the base field $k$ by a Dedekind scheme $S$. First of all, to prevent any ambiguity, throughout the whole paper by Dedekind scheme we mean a normal locally Noetherian connected scheme of dimension $0$ or $1$ (e.g. the spectrum of $\mathbb{Z}$, the spectrum of a discrete valuation ring, the spectrum of a field, ... ). So we have a Dedekind scheme $S$, a scheme $X$ faithfully flat and of finite type over $S$, we fix a section $x\in X(S)$ and we wonder whether we can build a universal torsor \emph{dominating} all the \emph{finite torsors}  over $X$ (i.e. torsors whose structural group is a finite and flat $S$-group scheme). As the case $dim(S)=0$ has already been considered we only study here the case $dim(S)=1$. The question is as before: is the category, that we still denote by $\mathcal{F}(X)$, of finite pointed torsors over $X$ cofiltered? A first answer has been given in \cite{Ga}, but the proof contains a mistake which has been corrected in \cite{AEG}. We thus introduce here some ideas contained in  \cite{AEG}, to which we refer the reader for the complete proofs of the results stated in this section and much more. Whenever we have a $S$-scheme $T$, we will denote by $T_{\eta}$ its generic fibre. We first give an answer to the above question (this is the subject of \cite{AEG}, \S 4 and \S 5.1):

\begin{theorem}\label{teoCofilDed}
Let $S$, $X$ and $x$ like before. Let us moreover assume that one of the following assumptions is satisfied:
\begin{enumerate}

\item for every $s\in S$ the fibre $X_s$ is reduced; 
\item for every $z\in X\backslash X_{\eta}$ the local ring $\mathcal{O}_{X,z}$ is integrally closed.
\end{enumerate}
Then the category $\mathcal{F}(X)$ is cofiltered. 
\end{theorem}
\proof
The strategy being the same as for the field case, we take three objects of $\mathcal{F}(X)$  $(Y_i,G_i,y_i), i=0, ..., 2$ and two morphisms as follows
$$\xymatrix{(Y_1,G_1,y_1)\ar[rd] & & (Y_2,G_2,y_2)\ar[dl]\\ & (Y_0,G_0,y_0) & }$$
and we want to prove that this diagram has a fibre product in $\mathcal{F}(X)$. Unfortunately in general the triple given by $(Y_1\times_{Y_0}Y_2,G_1\times_{G_0}G_2 ,y_1\times_{y_0}y_2 )$ is not a good candidate as $G_1\times_{G_0}G_2$ may easily be not flat (though finite) over $S$. But we can hopefully do the following: we can consider $(Y_1\times_X Y_2,G_1\times_S G_2 ,y_1\times_S y_2 )$ (which is certainly a torsor), its generic fibre (which is a torsor over $X_{\eta}$), the natural closed immersion  $$Y_{1,\eta}\times_{Y_{0,\eta}}Y_{2,\eta}\hookrightarrow Y_{1,\eta}\times_{X_{\eta}}Y_{2,\eta}$$ (which is a morphism of torsors, we simply omitted the structural groups and the marked points to ease the notation) and we need to hope that the Zariski closure of $Y_{1,\eta}\times_{Y_{0,\eta}}Y_{2,\eta}$ in $Y_1\times_X Y_2$ is a torsor under the action of the $S$-finite and flat group scheme obtained as Zariski closure of $G_{1,\eta}\times_{G_{0,\eta}}G_{2,\eta}$ in $G_1\times_S G_2$; it turns out to be the good candidate we were looking for and this follows from Lemma \ref{Lemma22}.
\endproof

\begin{lemma}\label{Lemma22} Notations are those of Theorem \ref{teoCofilDed}. Let us moreover assume that one of the assumptions given in Theorem \ref{teoCofilDed} is satisfied. Let $(Y,G,y)$ be an object of $\mathcal{F}(X)$ and $(T,H,t)$ an object of $\mathcal{F}(X_{\eta})$ contained in  $(Y_{\eta},G_{\eta},y_{\eta})$ (i.e. we are considering a morphism $(T,H,t)\to (Y_{\eta},G_{\eta},y_{\eta})$ whose induced morphism on the group schemes $H\to G_{\eta}$ is a close immersion). Then the triple $(\overline{T},\overline{H},\overline{t})$ obtained as the Zariski closure of $(T,H,t)$ in $(Y_{\eta},G_{\eta},y_{\eta})$ is an object of $\mathcal{F}(X)$.
\end{lemma}

Exactly like before we thus have a fundamental group scheme $\pi(X,x)$ of $X$ in $x$ and a universal $\pi(X,x)$-torsor over $X$ dominating, by a unique morphism, every finite pointed torsor over $X$.

Working over a base scheme with dimension $>0$ gives us a lot of freedom that we do not have when we are over a field: for example dealing with affine $S$-group schemes we immediately notice that there are many \emph{group objects} over $S$ whose fibres are finite group schemes. Of course there are finite and flat group schemes, which we already considered, but quasi-finite and flat group schemes (not necessarily finite) can also be considered. Finally, if we are brave enough, we can also consider \emph{non} flat quasi-finite group schemes, but for this we need to wait until \S \ref{sez:ogni}. We now conclude this section with a brief overview on the ``quasi-finite and flat'' world: again we have a Dedekind scheme $S$, a scheme $X$ faithfully flat and of finite type over $S$, we fix a section $x\in X(S)$ and we wonder whether we can build a universal torsor \emph{dominating} all the \emph{quasi-finite torsors}  over $X$ (i.e. torsors whose structural group is an affine quasi-finite and flat $S$-group scheme).  We now have a new, but similar, question: is the category, that we denote by $\mathcal{Q}f(X)$, of quasi-finite pointed torsors over $X$ cofiltered? At the end of this section we will explain why it can be useful to work in this new setting; first we state an existence result: 

\begin{theorem}\label{teoCofilDed2}
Let $S$, $X$ and $x$ like before. Let us moreover assume that $X$ is integral and normal and that for each $s\in S$ the fibre $X_s$ is normal and integral. Then the category $\mathcal{Q}f(X)$ is cofiltered. 
\end{theorem}
\proof
The proof is quite long and though the strategy is similar to that of Theorem \ref{teoCofilDed}, the details are a bit more complicated and we refer the reader to  \cite{AEG}, \S 5.2.
\endproof

Again, not surprisingly, we have a fundamental group scheme $\pi^{\text{qf}}(X,x)$ of $X$ in $x$ and a universal $\pi^{\text{qf}}(X,x)$-torsor over $X$ dominating, by a unique morphism, every quasi-finite pointed torsor over $X$. The reason why $\pi^{\text{qf}}(X,x)$ can be useful in many situations is the following: every finite torsor is, in particular, quasi-finite, so there is a natural morphism 
$$\pi^{\text{qf}}(X,x)\to \pi(X,x)$$ which is a schematically dominant morphism (i.e. the dual morphism on the coordinates rings is injective) but in general not an isomorphism. So in some sense $\pi^{\text{qf}}(X,x)$ is \emph{bigger} than $\pi(X,x)$ and, for a fixed point $s\in S$, the morphism  
$$\pi(X_s,x_s)\to \pi^{\text{qf}}(X,x)_s$$ is more likely to be an isomorphism than 
$$\pi(X_s,x_s)\to \pi(X,x)_s $$  it can thus be easier for $\pi^{\text{qf}}(X,x)$, though not easy, to obtain information about it from its fibres.

\section{Any dimension}
\label{sez:ogni}
We have seen that when the base scheme $S$ has dimension 1 the whole picture becomes more complicated. But there is at least a very useful property that we used a lot: over a Dedekind domain a finitely generated module is flat if and only if it is torsion free. This is no longer true when $dim(S)\geq 2$ so to build a fundamental group scheme classifying all the (quasi-)finite torsors might be very hard. From now until the end of the paper, unless stated otherwise, $S$ will be a locally Noetherian regular scheme, thus in particular when $dim(S)=0$ (resp. $dim(S)=1$) then $S$ is the spectrum of a field (resp. a Dedekind scheme). What we are suggesting in the reminder of these short notes is a new approach to study the existence of the fundamental group scheme for such a base scheme $S$. As fqpc torsors can be difficult to study globally in this new setting we first introduce new torsors, for a new Grothendieck topology, whose description can be a bit cumbersome, but they will globally behave well. Usual fpqc torsors will be a particular case. Let us do it by steps. 

\subsection{The bumpy topology}
\label{sez:BumpyTop}
We are going to define a new Grothendieck topology (see for instance \cite[\S 2.3]{Vi} for a friendly introduction). Let $f:Y\to X$ 
be a faithfully flat morphism of schemes. We recall that $f$ is called a \emph{fpqc} morphism if every quasi-compact open subset of 
$X$ is the image of a quasi-compact open subset of $Y$. In \cite[Proposition 2.33]{Vi} 
it is possible to find many equivalent definitions for $f$ to be a fpqc morphism. We first need to introduce a new class of
morphisms which will play the role of coverings in our Grothendieck topology. 

\begin{definition}
%Let $S$ be a scheme. 
A $S$-morphism $f:Y\to X$ is called \emph{bumpy} if for every $s\in S$ 
the induced morphism $f_s:Y_s\to X_s$ is fpqc whenever $X_s$ is not empty.
\end{definition}

\begin{proposition}\label{propVIST}
%Let $S$ be a scheme
 Let $f:Y\to X$ be an  $S$-morphism. The following are true:
\begin{enumerate}[(i)]
 \item If $f$ is fpqc then it is bumpy.
  \item If there is an open covering $\{V_i\}$ of $X$ such that the restriction $f^{-1}(V_i)\to V_i$ is bumpy then $f$ is bumpy.
 \item The composite of bumpy morphisms is bumpy.
 \item Let $i:U\to X$ be any morphism of schemes. If $f$ is bumpy then $i^*(f):Y\times_XU\to U$ is bumpy.
 \item If $f:Y\to X$ is of finite presentation and for every $s\in S$ 
the induced morphism $f_s:Y_s\to X_s$ is faithfully flat then $f$ is bumpy.
\end{enumerate}

\end{proposition}
\begin{proof}
 These are mainly easy consequences of the definition and \cite[Proposition 2.35]{Vi}.
\end{proof}

\begin{definition}
\label{defBumpyTop} The \emph{bumpy} topology on the category $(Sch/S)$ is the topology in which the coverings $\{U_i\to U\}$ 
are collections of morphisms such that the induced morphism $\coprod_i U_i\to U$ is bumpy (here $\coprod_i U_i$ denotes the disjoint 
union).
\end{definition}

It is indeed easy to verify that this 
defines a Grothendieck topology (see \cite[Definition 2.24]{Vi}). This gives rise to the bumpy site $(Sch/S)^{bp}$. 
It will be useful to observe that a single bumpy morphism is thus a covering. 

\begin{remark}
Every affine $S$-group scheme $G\to S$ of finite type is bumpy. Indeed here $X=S$ and  for all $s\in S$ the restriction $f_s:G_s\to Spec(k(s))$ is fpqc.
\end{remark}

We need a new definition of torsor whose difference will be clear from the very beginning, that is even from the concept of trivial 
torsor.

\begin{definition}\label{defTorsTriv}
A trivial torsor consists of a $S$-scheme $Y$ with an action 
$\sigma:Y\times G\to Y$  together with an invariant arrow $f:Y\to X$, such that there is a $G$-equivariant $X$-morphism
$\varphi : X\times_S G\to Y$ with the property that for every $s\in S,$ $\varphi_s : X_s \times_{k(s)} G_s\to Y_s$
 is an isomorphism such that $f_s\circ \varphi_s=pr_1$ . 
\end{definition}

\begin{definition}
Let $S$ be a scheme and $G$ an affine $S$-group scheme. A bumpy $G$-torsor is a $S$-scheme $Y$ with an action 
$\sigma:Y\times G\to Y$ and a $G$-invariant $S$-morphism of schemes (i.e.  $f\circ \sigma = f \circ pr_1$ where 
$pr_1:Y\times G\to Y$ is the first projection), such that there exists a covering $\{U_i\to X\}$ in the bumpy site $(Sch/S)^{bp}$  
with the property that for each $i$ the arrow $pr_1:U_i\times_{X}Y\to U_i$ is a trivial torsor. 
\end{definition}

Every $G$-invariant morphism $Y\to X$ which is a torsor for the fpqc topology is also a torsor for the bumpy topology. Reciprocally a flat morphism  $Y\to X$ which is a bumpy torsor for a flat $S$-group scheme $G$ is nothing but a fpqc morphism (by the \emph{crit\`ere de platitude par fibres} the morphism $\varphi$ in Definition \ref{defTorsTriv} becomes an isomorphism) so in particular a fpqc torsor.
From now on we will simply say \emph{torsor} instead of \emph{bumpy torsor} unless we will need to stress the difference.  
Let us assume that for an affine $G$-invariant bumpy morphism $Y\to X$ the 
canonical morphism $Y\times_S G\to Y\times_X Y$ is a fibrewise isomorphism, that is for every $s\in S$ it induces an isomorphism $Y_s\times_{k(s)} G_s\to Y_s\times_{X_s} Y_s$. Since $f:Y\to X$ is itself a 
covering in the bumpy site $(Sch/S)^{bp}$ then this implies that $f$ is a $G$-torsor. More generally we have the following:

\begin{proposition}\label{propTorsCrit}
 Let $S$ be a scheme and $G$ an affine $S$-group scheme. Let $Y$ be a $S$-scheme with an action 
$\sigma:Y\times G\to Y$ and $f:Y\to X$ an affine bumpy $G$-invariant $S$-morphism. Then $f:Y\to X$ is a $G$-torsor if and only if 
for all $s\in S$ the canonical morphism $Y\times_SG\to Y\times_XY$ is a fibrewise isomorphism, that is the induced morphism $(Y\times_SG)_{s}\to (Y\times_XY)_{s}$ is an isomorphism.
\end{proposition}

\begin{proof}
On one direction it follows from previous discussion. On the other direction
let us assume that  $f:Y\to X$ is a $G$-torsor then there exists a covering $\{U_i\to X\}$ in the bumpy site 
 $(Sch/S)^{bp}$ such that for each $i$ the arrow $pr_1:U_i\times_{X}Y\to U_i$ is a trivial torsor, which implies that for every $s\in S,$  $U_{i,s}\times_{X_s}Y_s$ is isomorphic to  ${U_{i,s}}\times_{k(s)}G_s$. Now  $\{U_{i,s}\to X_s\}$ is a fpqc covering so the previous isomorphism says that the canonical induced morphism $Y_s\times_{k(s)}G_s\to Y_s\times_{X_s}Y_{s}$ is an isomorphism, which is enough to conclude. 
% which is a 
% bumpy morphism in the category $(Sch/S)$ (Proposition \ref{propVIST}, (iii)). Hence if we set $V:=\coprod_i U_i$ then 
% $Y_V:=V\times_{X}Y\to V$ is isomorphic to $G_{V}\to V$ which is a bumpy morphism in the category $(Sch/S)$. Now, 
% for every $s\in S$ we have the cartesian square
% $$\xymatrix{{(Y_V)}_s \ar[r]\ar[d] & Y_s \ar[d] \\ V_s \ar[r] & X_s  }$$
% and since ${(Y_V)}_s\to V_s$ and $V_s \to X_s$ are fpqc then $V_s \to X_s$ is fpqc too by faithfully flat descent. It is also clear 
%that $(Y\times_SG)_{V}\to (Y\times_XY)_{V}$ is an isomorphism. Reducing to $s$ for all $s\in S$ this proves, by faithfully flat 
%descent that $(Y\times_SG)_{s}\to (Y\times_XY)_{s}$ is an isomorphism.
\end{proof}

We can restate Proposition \ref{propTorsCrit} saying that
for an affine $S$-group scheme $G$,  a $S$-scheme $Y$ with an action 
$\sigma:Y\times G\to Y$ and an affine bumpy $G$-invariant $S$-morphism, $f:Y\to X$ is a $G$-torsor if and only if $f_s:Y_s\to X_s$ is a $G_s$-torsor for all $s\in S$.

From now on, and unless stated otherwise, by quasi-finite morphism we simply mean of finite type and \emph{with finite fibers}. An affine quasi-finite (non flat) group scheme over $S$ can have a fiber at a closed point of order higher then the generic fiber (or other closed points), whence the adjective \textit{bumpy} which stress the non-flatness of certain morphisms. Hereafter some examples 
of such group schemes (in order to have non-trivial examples, $S$ clearly cannot be the spectrum of a field).

\begin{example}\label{exMau}
Let $R$ be a discrete valuation ring with fraction and residue field respectively denoted by $K$ and $k$. We assume the latter to 
be of positive characteristic $p$. The letter $\pi$ will denote an uniformising element. We set $$G:=Spec \frac{R[x]}{x^p-x, \pi x}
\qquad \text{and} \qquad H:=Spec \frac{R[x]}{x^p, \pi x};$$
they are quasi finite $R$-group schemes of finite type when provided with comultiplication, counit and coinverse given respectively by $\Delta(x):=x\otimes 1+1\otimes x; \varepsilon(x):=0; S(x):=-x$. We immediately observe that neither $G$ nor $H$ are $S$-flat: indeed $x$ is, in both cases, a $R$-torsion element. Their generic fibers are trivial $K$-group schemes while
the special fibers $G_s$ and $H_s$ (of $G$ and $H$ respectively) are isomorphic to $(\mathbb{Z}/p\mathbb{Z})_k$ 
and $\alpha_{p,k}$ respectively. When $char(K)=p$ then $G$ can be easily recovered as the kernel of the morphism 
$(\mathbb{Z}/p\mathbb{Z})_R\to \mathcal{M}^1$ (sending $x\mapsto \pi x$) where $\mathcal{M}^n:= Spec \frac{R[x]}{x^p-\pi^{n(p-1)} x}$ 
is the finite and flat $R$-group scheme defined in \cite[\S 3.2]{Ma}, so in particular $G$ is finite and not just quasi-finite. In 
a similar way $H$ can be  recovered as the kernel of the morphism $\mathcal{M}^1\to \mathcal{M}^2$ (sending $x\mapsto \pi x$). 
\end{example}

\begin{remark}\label{remAccaIJ}
Let $R$ be a discrete valuation ring of positive equal characteristic $p$. Let $\mathcal{M}^n$ be as in Example \ref{exMau}.
We define $H_{ij}:=ker (\varphi:\mathcal{M}^i\to \mathcal{M}^j)$ (where $\varphi^{\#}:x\mapsto \pi^{j-i} x)$). We observe, 
for instance,   that there is a natural group scheme morphism $u:H_{12}\to H_{13}$ whose corresponding morphism on the 
coordinate rings is given by $u^{\#}:\frac{R[x]}{x^p, \pi^2 x}\to \frac{R[t]}{t^p, \pi t}$, $x\mapsto t$. It is immediate to observe 
that for all $s\in Spec(R)$ the restriction $u_s:H_{12,s}\to H_{13,s}$ is an isomorphism though $u$ is not an isomorphism: indeed 
$u^{\#}$ is not injective. This phenomenon cannot happen in the flat schemes world.
\end{remark}

\begin{remark}
Notations being as in Remark \ref{remAccaIJ} we observe that $H_{13}$ can be both seen as a trivial $H_{13}$-torsor and as a 
trivial $H_{12}$-torsor. This phenomenon is not entirily new as even in the easiest case where $X=S$ is the spectrum of a positive
characteristic field then a $\mu_p$-trivial torsor can also be seen as a $\alpha_p$-trivial torsor.
\end{remark}

%\begin{example}
%Let $\mathbb{Z}$ be the ring of integers and let $p,q$ be two prime numbers. We consider the $\mathbb{Z}$-scheme (quasi-finite and of 
%finite type, but not flat) $$N:=Spec \frac{\mathbb{Z}[x,y]}{x^2-x, px, y^2-y, qy.} $$ It becomes a $\mathbb{Z}$-group scheme 
%with comultiplication, counit and  coinverse given by $\Delta(x):=x\otimes 1+1\otimes x; \Delta(y):=y\otimes 1+1\otimes y; 
%\varepsilon(x):=0; \varepsilon(y):=0; S(x):=-x; S(y):=-y$. Here we notice that the generic fibre $N_{\eta}=Spec(\mathbb{Q})$ 
%is trivial, then 
%$N_p\simeq (\mathbb{Z}/2\mathbb{Z})_{\mathbb{F}_p}$, $N_q\simeq (\mathbb{Z}/2\mathbb{Z})_{\mathbb{F}_q}$ and for any prime $p_i$, 
%such that $p\neq p_i \neq q$ then $N_{p_i}\simeq (\mathbb{Z}/2\mathbb{Z})_{\mathbb{F}_{p_i}}\times_{Spec(\mathbb{F}_{p_i})} 
%(\mathbb{Z}/2\mathbb{Z})_{\mathbb{F}_{p_i}}$.
%\end{example}

These short examples show that it is easy and very natural to encounter non flat quasi-finite group schemes.

\subsection{The bumpy fundamental group scheme}

Let $S$ be any scheme and $X\to S$ a faithfully flat morphism of schemes of finite type. We assume the existence of a section $x\in X(S)$. Even in this 
new setting by pointed torsor we will mean a $G$-torsor $Y\to X$ endowed with a section $y\in Y_x(S)$. Let $Z\to X$ be a $H$-torsor
pointed in $z\in Z_x(S)$; a morphism of pointed torsors from $Y\to X$ to $Z\to X$ is the data of two morphisms $\alpha:G\to H$ and
$\beta: Y\to Z$ where $\alpha$ is a morphism of group schemes, $\beta(y)=z$ and  the following diagram commutes:
$$\xymatrix{Y\ar[r]^{\sigma_Y}\ar[d]_{\beta} & Y\times G \ar[d]^{\beta\times \alpha} \\ Z \ar[r]_{\sigma_Z} & Z\times H}$$
where ${\sigma_Y}$ and ${\sigma_Z}$ denote the actions of $G$ on $Y$ and $H$ on $Z$ respectively. Though in \S \ref{sez:BumpyTop} we gave a very general definition for \textit{non flat} torsors we do not need to work in that generality: our purpose is indeed to work in the smallest possible cofiltered full subcategory of the category of bumpy pointed torsors containing fpqc torsors; since we do not know, at this stage, whether the category of (quasi-)finite fpqc torsors itself is cofiltered we suggest a different candidate and we will prove that it is a cofiltered category. This is why from now on we only consider the following situation:
\begin{definition}\label{definTORSQF}
Let $G\to S$ be a quasi finite, affine group scheme of finite type and $Y\to X$ a bumpy $G$-torsor. Such a quasi-finite torsor will be often denoted by $(Y,G,y)$.   We will denote by $\mathcal{Q}f(X,x)$ the category of all those quasi-finite pointed torsors which also satisfy the following property: the canonical fibrewise isomorphism $Y\times_S G\to Y\times_X Y$ is (globally) an isomorphism. The full subcategory of finite pointed torsors (i.e. $G\to S$ is finite) will be denoted by $\mathcal{F}(X,x)$.
\end{definition}

The following theorem holds even without the \textit{stronger} assumption on torsors made in \ref{definTORSQF}, but, again, we stress that not only we do not need that generality but we also prefer to restrict as much as possible our category to get closer to the category of fpqc torsors.

\begin{theorem}\label{teoCOFILTAGAIN} Let $X$ be a scheme and $X\to S$ a faithfully flat morphism of finite type with the property that
for all $s\in S$ the fiber $X_s$ is reduced and connected. Let $x\in X(S)$ be a section.  Then the categories $\mathcal{Q}f(X,x)$  and $\mathcal{F}(X,x)$ are cofiltered.  
\end{theorem}
 \begin{proof}
  The proof for the two categories being essentially the same here we only consider  $\mathcal{Q}f(X,x)$. Since $\mathcal{Q}f(X,x)$ has a final object, it is sufficient to prove that given three objects 
  $(Y_i, G_i, y_i), i=0, 1 , 2$ of $\mathcal{Q}(X,x)$ and two morphisms 
$\varphi_i:(Y_i, G_i, y_i)\to (Y_0, G_0, y_0), i=1, 2$, there exists a forth object $(Y_3, G_3, y_3)$ and two morphisms 
$\psi_i:(Y_3, G_3, y_3)\to (Y_i, G_i, y_i), i=1, 2$ that are closing the diagram. We simply take the fiber product $(Y_3, G_3, y_3)
:= (Y_2\times_{Y_0}Y_1,G_2\times_{G_0}G_1, y_2\times_{y_0}y_1).$ We need to prove that it belongs to $\mathcal{Q}f(X,x)$. The assumption that 
for all $s\in S$ the fiber $X_s$ is reduced and connected is required to ensure that the category $\mathcal{Q}(X_s,x_s)$ 
of finite pointed (fpqc) torsors is cofiltered. For this reason the fiber product $Y_3$ is bumpy over $X$. 
It is now immediate to observe that  $Y_3\to X$ is a bumpy $G_3$-torsor following Proposition \ref{propTorsCrit}. This is however  not sufficient to conclude as we still need to prove that the fiberwise isomorphism $Y_3\times_S G_3\to Y_3\times_X Y_3$ is actually an isomorphism:  it is enough to pull back over $Y_3$ the square given by $Y_3, Y_2, Y_1, Y_0$ and the conclusion follows from the fact that $Y_2, Y_1, Y_0$ belong to $\mathcal{Q}f(X,x)$.
 \end{proof}

 \begin{corollary} Let $X\to S$ be a surjective morphism of finite type with the property that for all $s\in S$ the fiber $X_s$ is reduced and connected. Then there exist a $S$-group scheme $\pi^{Bqf}(X,x)$ and a  (pointed in $x^{Bqf}$ over $x$)  $\pi^{Bqf}(X,x)$-torsor $X^{Bqf}\to X$ universal in the following sense: for every object $(Y,G,y)$ of $\mathcal{Q}f(X,x)$ there is a unique $X$-morphism of (pointed) torsors $X^{Bqf}\to Y $.
 
\end{corollary}

\begin{proof}Very naturally from Theorem \ref{teoCOFILTAGAIN} we have, writing shortly, 
$$( X^{Bqf},\pi^{Bqf}(X,x), x^{Bqf}):= \varprojlim_i (Y_i, G_i, y_i)$$ the 
limit running through all the objects of $\mathcal{Q}f(X,x)$, which exists as the morphisms are affine. 
\end{proof}

 \begin{definition}
  We call $\pi^{Bqf}(X,x)$ the bumpy fundamental group scheme. In a similar way we obtain $$( X^{B},\pi^{B}(X,x), x^{B}):= \varprojlim_i (Y_i, G_i, y_i)$$
the limit running through all the objects of $\mathcal{F}(X,x)$. This triple is universal in the obvious sense already mentioned. The $S$-affine group scheme $\pi^{B}(X,x)$ will be called the finite bumpy fundamental group scheme.
 \end{definition}
 
It is worth repeating that it is not known whether the fpqc pointed finite (resp. quasi-finite) torsors form a cofiltered category for such a general base scheme $S$, but \emph{all} of them are already inside $\mathcal{F}(X,x)$ (resp. $\mathcal{Q}f(X,x)$) and so provided the fundamental group scheme $\pi(X,x)$ classifying all the finite pointed fpqc torsors (resp. the quasi-finite fundamental group scheme $\pi^{\text{qf}}(X,x)$  classifying all the quasi-finite pointed fpqc torsors) exists, there would be a natural morphism $\pi^{B}(X,x)\to \pi(X,x)$ (resp. $\pi^{Bqf}(X,x)\to \pi^{\text{qf}}(X,x)$).

\section{Over a Dedekind scheme (revisited)}
\label{sez:revisited}
This section is meant to show that our construction is not esoteric: indeed the bumpy fundamental group scheme of a scheme $X$ coincides with the already existing ones, whenever comparable. It is immediate from the definition of bumpy torsor that the bumpy fundamental group scheme of a scheme $X$ defined over a field is nothing but the fundamental group scheme of Nori, described in these notes in \S \ref{sez:campo}. Here we show that if the base scheme $S$ is a Dedekind scheme then, when comparable, the bumpy fundamental group scheme of $X$ coincides with the ``usual'' fundamental group scheme described in \S \ref{sez:Dede}. The reason is that in this
case we can prove (and we actually will in few lines) that every quasi-finite bumpy torsor is preceded by a quasi-finite fpqc torsor. And, also,
every finite bumpy torsor is preceded by a finite fpqc torsor. 
 
\begin{theorem}\label{teoLAST}
Let $S$ be a Dedekind scheme, $X$ a faithfully flat $S$-scheme of finite type and $x\in X(S)$ a section. Let us moreover assume that one of the following assumptions is satisfied:

\begin{enumerate}
\item for every $s\in S$ the fibre $X_s$ is reduced; 
\item $X$ is integral and normal and that for each $s\in S$ the fibre $X_s$ is normal and integral
\end{enumerate}
then if (1) is satisfied the natural morphism $\pi^{B}(X,x)\to \pi(X,x) $ is an isomorphism and if moreover (2) is satisfied the natural morphism $\pi^{Bqf}(X,x)\to \pi^{\text{qf}}(X,x) $ is an isomorphism too.
\end{theorem}
\proof The proof is similar to (and can be deduced from) \cite{AEG}, Proposition 4.2, Proposition 5.2, Proposition 5.5 and we leave to the reader all the details. Here we only prove that $\pi^{B}(X,x)\to \pi(X,x) $ is an isomorphism when $S=Spec(R)$ where $R$ is a discrete valuation ring (this is thus a subcase of case (1)): it is sufficient to prove that any object of $\mathcal{F}(X,x)$ is preceded by a finite fpqc torsor, that means that if $Y\to X$ is a finite bumpy $G$-torsor then there exist a finite and flat $R$-group scheme $H$ and a finite fpqc $H$-torsor $Z\to X$ and a morphism (of pointed torsors) $Z\to Y$. So let $Y_{\eta}$ be, as usual, the generic fibre of $Y$: it is a finite fpqc $G_{\eta}$-torsor over  $X_{\eta}$. We claim that the Zariski closure $\overline{Y_{\eta}}$ of $Y_{\eta}$ in $Y$ is a fpqc-torsor over $X$ under the action of the $R$-finite and flat group scheme obtained as the Zariski closure $\overline{G_{\eta}}$ of $G_{\eta}$ in $G$. Following \cite{EGAIV-2} (2.8.3) we deduce an action  $\overline{Y_{\eta}}\times_S\overline{G_{\eta}}\to \overline{Y_{\eta}}$ 
compatible with the action of $G$ on $Y$. Thus, in particular, the canonical morphism 
$$u: \overline{Y_{\eta}}\times_S\overline{G_{\eta}}\to \overline{Y_{\eta}}\times_X \overline{Y_{\eta}}$$ complete the following diagram:

$$\xymatrix{\overline{Y_{\eta}}\times_S\overline{G_{\eta}}\ar[r]^u\ar@{^{(}->}[d]_i & \overline{Y_{\eta}}\times_X \overline{Y_{\eta}} 
\ar@{^{(}->}[d]^j\\ Y\times_S G \ar[r]^v & Y\times_X Y }$$

where $v$ is an isomorphism, $i$ and $j$ are closed immersion so $u$ is a closed immersion too.

The huge difference between $G$ and $\overline{G_{\eta}}$ is that $G$ may not be $R$-flat, while $G$ is certainly $R$-flat. Thus according to \cite{SGA3} Expos\'e V, Théorème 7.1, the quotient  $\overline{Y_{\eta}}/\overline{G_{\eta}}$ exists and it is represented by a scheme such that  $p:\overline{Y_{\eta}} \to \overline{Y_{\eta}} /\overline{G_{\eta}}$ is faithfully flat and the morphism $\overline{Y} \to X$ factors through  $\lambda:\overline{Y_{\eta}}/\overline{G_{\eta}}\to X$ that we study below. First we observe that $\lambda$ is separated: since  $p:\overline{Y_{\eta}} \to \overline{Y_{\eta}} /\overline{G_{\eta}}$ is proper we just need to consider the following commutative diagram:
$$\xymatrix{\overline{Y_{\eta}}\ar@{^{(}->}[r]^{\Delta_2}\ar@{->>}[d]_{p} & \overline{Y_{\eta}}\times_X \overline{Y_{\eta}} \ar@{->>}
[d]^{p\times p}\\ \overline{Y_{\eta}}/\overline{G_{\eta}} \ar[r]^{\Delta_1} & \overline{Y_{\eta}}/\overline{G_{\eta}}\times_X \overline{Y_{\eta}}/\overline{G_{\eta}} }$$
where $\Delta_2$ is a closed immersion as $\overline{Y_{\eta}}\to X$ is propre (by assumption, cf. Definition \ref{definTORSQF}, $\overline{Y}\to X$ is propre). This diagram implies  $$\Delta_1(\overline{Y_{\eta}}/\overline{G_{\eta}})=\Delta_1(p(\overline{Y_{\eta}}))=(p\times p) (\Delta_2(\overline{Y_{\eta}})) $$ 
which is closed in $\overline{Y_{\eta}}/\overline{G_{\eta}}\times_X \overline{Y_{\eta}}/\overline{G_{\eta}}$ because $(p\times p ) \Delta_2$ is propre. So finally $\lambda$ is separated (\cite{H77}, II, Corollary 4.2). But $p$ is  surjective,  $\lambda \circ p$ is proper (it is actually the composition of  $Y\to X$ and the closed immersion $\overline{Y_{\eta}} \hookrightarrow Y$); hence $\lambda$ is  propre too (\cite{Liu}, Ch 3, Prop 3.16 (f)). Furthermore we observe that for every $x\in X$ the set
$\lambda^{-1}(x)$ is finite so $\lambda$, which is proper and quasi-finite, is in particular finite (\cite{EGAIV-3} Th\'eor\`eme 8.11.1), hence affine. We did not finish yet, as we want to prove that $\lambda$ an isomorphism we first start by the surjectivity: the image of $\lambda $ in $X$ contains the image of $\lambda \circ p$ which is dense in  $X$; furthermore, the properness of $\lambda $ ensures that $\lambda (\overline{Y_{\eta}} /\overline{ G_{\eta}})$ is closed in $X$. We deduce from this the equality $\lambda (\overline{Y_{\eta}} /\overline{G_{\eta}})=X$. In order to conclude that $\lambda$ is an isomorphism we need to observe that  $\lambda _s:(\overline{Y_{\eta}}/\overline{G_{\eta}})_s\to X_s$ is surjective too, but $X_s$ being reduced then $\lambda_s^{\#}:\mathcal{O}_{X_s}\hookrightarrow {(f_s)}_*(\mathcal{O}_{\overline{Y_{\eta}}/\overline{G_{\eta}}})$ is injective (\cite{EGAI}, Corollaire 1.2.7); according to \cite{WWO}, Lemma 1.3, the morphism $\lambda$, which is affine and is an isomorphism generically, is an isomorphism globally. This proves that $\overline{Y_{\eta}}\to X$ is a fpqc finite $\overline{G_{\eta}}$-torsor. It is of course pointed if $Y$ is pointed which concludes the proof once we set $H:= \overline{G_{\eta}}$ and $Z:= \overline{Y_{\eta}}$.
\endproof

In particular when $S$ is a Dedekind scheme $\pi^{B}(X,x)$ and $\pi^{Bqf}(X,x)$ (for $X$ as in the statement of Theorem ) are $S$-flat. We do not know if this remains true when $S$ has higher dimension.

%\begin{corollary}\label{corLAST}
%\end{corollary} 


\begin{thebibliography}{99}

\bibitem[AEG15]{AEG} {\sc M. Antei, M. Emsalem, C. Gasbarri}, Sur l'existence du schéma en groupes fondamental, arXiv:1504.05082v2 [math.AG] 

\bibitem[BoVi]{BV} {\sc N. Borne, A. Vistoli }, \emph{The Nori fundamental gerbe of a fibered category}, Journal og Algebraic Geometry (to appear).

\bibitem[DeMi82]{DelMil} {\sc P.~Deligne, J. S.~Milne}, \emph{Tannakian
Categories}, in \emph{Hodge Cycles, Motives, and Shimura
Varieties}, Lectures Notes in Mathematics 900, Springer-Verlag,
(1982), 101-227.

\bibitem[SGA3]{SGA3} {\sc M. Demazure, A. Grothendieck} \emph{Sch\'emas en groupes. I: Propri\'et\'es g\'en\'erales des sch\'emas en groupes}. S\'eminaire de G\'eom\'etrie Alg\'ebrique du Bois Marie 1962/64 (SGA 3). Lecture Notes in Mathematics, Vol. 151 Springer-Verlag, Berlin-New York 1970


\bibitem[EGAI]{EGAI} {\sc A. Grothendieck}, \emph{\'El\'ements de  g\'eom\'erie alg\'ebrique}. I. \emph{Le langage des sch\'emas}. Publications Math\'ematiques de l'IH\'ES, 4, (1960).


\bibitem[EGAIV.2]{EGAIV-2} {\sc A. Grothendieck}, \emph{\'El\'ements de  g\'eom\'erie alg\'ebrique}. IV. \emph{\'Etude locale des sch\'emas et des morphismes
de sch\'emas}. 2, Publications Math\'ematiques de l'IH\'ES, 24, (1965).

\bibitem[EGAIV.3]{EGAIV-3} {\sc A. Grothendieck}, \emph{\'El\'ements de  g\'eom\'erie alg\'ebrique}. IV. \emph{\'Etude locale des sch\'emas et des morphismes
de sch\'emas}. 3, Publications Math\'ematiques de l'IH\'ES, 28 (1966). 

\bibitem[Ga03]{Ga} {\sc C. Gasbarri}, Heights of vector bundles and the fundamental group scheme of a curve.
\textit{Duke Math. J.} \textbf{117} (2003), no. 2, 287--311.

\bibitem[SGA1]{SGA1} {\sc A. Grothendieck}, \emph{Rev\^etements \'etales et groupe fondamental}, S\'eminaire de g\'eom\'etrie alg\'ebrique du Bois Marie, (1960-61).

\bibitem[Ha77]{H77} {\sc R. Hartshorne}, \emph{Algebraic Geometry}, GTM, Springer Verlag (1977).

\bibitem[Li02]{Liu} {\sc Q. Liu}, \emph{Algebraic geometry and arithmetic curves},
Oxford Science Publications (2002).

\bibitem[Ma03]{Ma} {\sc S. Maugeais}, Rel\`evement des rev\^etements $p$-cycliques des courbes rationnelles semi-stables, 
\textit{ Math. Ann.} \textbf{327} (2003), no. 2, 365--393.

\bibitem[No76]{N1} {\sc M. V. Nori}, On the representations of the
fundamental group, \textit{Compos. Math.} \textbf{33} (1976), 29--42.

\bibitem[No82]{N2} {\sc M. V. Nori}, The fundamental group-scheme, \textit{Proc.
Ind. Acad. Sci. (Math. Sci.)} \textbf{91} (1982), 73--122.

\bibitem[Sz09]{Sza} {\sc T. Szamuely}, \emph{Galois groups and fundamental groups}. Cambridge Studies in Advanced Mathematics, 117. Cambridge University Press, Cambridge, 2009.


\bibitem[Vi05]{Vi} {\sc  A. Vistoli},  Grothendieck topologies, fibered categories and descent theory. Fundamental algebraic geometry, 
1--104, 
\textit{Math. Surveys Monogr., 123, Amer. Math. Soc., Providence, RI}, 2005. 

\bibitem[Wa79]{Wa} {\sc W. C. Waterhouse},  Introduction to affine group schemes. \textit{Graduate Texts in Mathematics, 66. Springer--Verlag, 
New York--Berlin}, 1979

\bibitem[WW80]{WWO}  {\sc  W.C. Waterhouse, B. Weisfeiler}, \emph{One-dimensional affine group
schemes}, Journal of Algebra, 66, 550-568 (1980).



\end{thebibliography}
\end{document}